\documentclass[12pt,reqno]{article}
mm \textheight=8.9in

\usepackage{amssymb}
\usepackage{amsmath}
\usepackage{amsthm}

\usepackage{color}

\newcommand{\br}[3]{{$#1$}$\lower4pt\hbox{$\tp\atop\raise4pt \hbox{$\scriptscriptstyle{#2}$}$} ${$#3$}}
\newcommand{\tw}[3]{{$#1$}${\,\scriptscriptstyle {#2}}\atop\raise9pt\hbox{$\scriptstyle\tp$} ${$#3$}}
\newcommand{\ttps}[2]{{#1}\raise5pt\hbox{$\lower12pt\hbox{$\scriptstyle\tp$}\atop \lower0pt\hbox{$\tilde\;$}$}\raise4.5pt\hbox{${\scriptstyle{#2}}$}}
\newcommand{\st}[1]{\mbox{${\,\scriptscriptstyle {#1}}\atop\raise5.5pt\hbox{$*$}$}}

\newcommand{\rd}[1]{\mbox{${\,\scriptscriptstyle {#1}}\atop\raise5.5pt\hbox{$\bullet$}$}}
\newcommand{\rt}[1]{\otimes_\chi}
\newcommand{\lt}[1]{\mbox{${\,\scriptscriptstyle {#1}}\atop\raise5.5pt\hbox{$\ltimes$}$}}
\newcommand{\btr}{\raise1.2pt\hbox{$\scriptstyle\blacktriangleright$}\hspace{2pt}}
\newcommand{\btl}{\raise1.2pt\hbox{$\scriptstyle\blacktriangleleft$}\hspace{2pt}}

\newcommand{\lcr}{\raise1.0pt \hbox{${\scriptstyle\rightharpoonup}$}}
\newcommand{\rcr}{\raise1.0pt \hbox{${\scriptstyle\leftharpoonup}$}}

\newcommand{\ttp}{{\lower12pt\hbox{$\tp$}\atop \hbox{$\tilde\;$}}}

\newcommand{\id}{\mathrm{id}}

\newcommand{\Hc}{\mathcal{H}}

\newcommand{\Sc}{\mathcal{S}}
\renewcommand{\S}{\mathcal{S}}
\newcommand{\Ru}{\mathcal{R}}

\newcommand{\Cc}{\mathcal{C}}

\renewcommand{\O}{\mathcal{O}}

\newcommand{\C}{\mathbb{C}}

\newcommand{\Qbb}{\mathbb{Q}}

\newcommand{\Z}{\mathbb{Z}}

\newcommand{\N}{\mathbb{N}}

\newcommand{\tp}{\otimes}

\newcommand{\zt}{\zeta}

\newcommand{\U}{U}

\newcommand{\ve}{\varepsilon}
\newcommand{\gm}{\gamma}
\newcommand{\dt}{\delta}

\newcommand{\op}{\oplus}
\newcommand{\la}{\lambda}

\newcommand{\End}{\mathrm{End}}

\newcommand{\rk}{\mathrm{rk}}

\newcommand{\Rm}{\mathrm{R}}

\newcommand{\ad}{\mathrm{ad}}

\newcommand{\Ga}{\Gamma}
\newcommand{\g}{\mathfrak{g}}
\newcommand{\e}{\mathfrak{e}}
\renewcommand{\b}{\mathfrak{b}}

\newcommand{\h}{\mathfrak{h}}

\newcommand{\s}{\mathfrak{s}}

\newcommand{\f}{\mathfrak{f}}
\renewcommand{\o}{\mathfrak{o}}

\newcommand{\eps}{\epsilon}

\newcommand{\p}{\mathfrak{p}}
\renewcommand{\l}{\mathfrak{l}}

\newcommand{\si}{\sigma}
\newcommand{\al}{\alpha}

\renewcommand{\b}{\mathfrak{b}}
\newcommand{\bt}{\beta}

\newcommand{\be}{\begin{eqnarray}}
\newcommand{\ee}{\end{eqnarray}}

\newtheorem{thm}{Theorem}[section]
\newtheorem{propn}[thm]{Proposition}
\newtheorem{lemma}[thm]{Lemma}

\newtheorem{remark}[thm]{Remark}
\newtheorem{definition}[thm]{Definition}

\newcount\prg

\newcommand{\parag}{\advance\prg by1 {\noindent\bf\thesection.\the\prg\hspace{6pt}}}

\begin{document}
\title{Factorization of Shapovalov elements}
\author{
Andrey Mudrov \vspace{10pt}\\
\\
\small
 St.-Petersburg Department of Steklov Mathematical Institute,\\
\small
27 Fontanka nab,
191023 St.-Petersburg, Russia,
\vspace{10pt}\\
\small
 Moscow Institute of Physics and Technology,\\
\small
9 Institutskiy per., Dolgoprudny, Moscow Region,
141701, Russia,
\vspace{10pt}\\
\small University of Leicester, \\
\small University Road,
LE1 7RH Leicester, UK,
\vspace{10pt}\\
\small e-mail:  am405@le.ac.uk
\\
}

\date{ }

\maketitle

\begin{abstract}

Shapovalov elements $\theta_{\bt,m}$ are special elements in a Borel subalgebra of a classical or quantum
universal enveloping algebra parameterized by a positive root $\bt$ and a positive integer $m$. They relate the
canonical generator of
a reducible Verma module with highest vectors of its Verma submodules.
For $m=1$, they can be explicitly obtained as matrix elements of the inverse Shapovalov form.
We extend this approach to  $m>1$ for all $\bt$ but three roots in $\g_2$, $\f_4$, and $\e_8$, presenting $\theta_{\bt,m}$ as
a product of matrix elements of weight $\bt$.
\end{abstract}

\\
{\small \underline{AMS classification codes}: 17B10, 17B37}
 \newpage
 \section{Introduction}
The Berstein-Gelfand-Gelfand representation category $\O$   of semi-simple Lie algebras and quantum groups \cite{BGG1} is
one of the fundamental concepts appearing  in various fields of mathematics and mathematical physics.
In particular, it accommodates finite-dimensional and numerous important infinite dimensional representations like  parabolic modules and
their  generalizations \cite{H}. There are distinguished objects in $\O$ called Verma modules that feature a universality
property: all irreducible modules in $\O$ are their quotients.   The maximal proper submodule
in a Verma module is  generated by extremal vectors \cite{BGG2}, which are invariants
of the positive triangular  subalgebra.
This makes extremal vectors critically important in representation theory.

Extremal vectors in a Verma module are related with a vacuum vector of highest weight  via special elements
$\theta_{\bt,m}$ of the negative Borel subalgebra that are
called Shapovalov elements \cite{Shap,C}. They are parameterized with a positive root $\bt$ and an integer $m\in \N$
validating, in the classical version, the Kac-Kazhdan condition $2(\la+\rho,\bt)-m(\bt,\bt)=0$ on the highest weight $\la$ (with $\rho$
being the half-sum of positive roots). This  condition guarantees that the Verma module is reducible. In the special case when the root $\bt$ is
simple, $\theta_{\bt,m}=f^m_\bt$, where $f_\bt$ is the corresponding Chevalley generator. For compound $\bt$,
the Shapovalov element is a polynomial in the simple root generators with coefficients in the Cartan subalgebra.

Factorization of $\theta_{\bt,m}$ to a product of polynomials of lower degree is convenient both
for their explicit construction and for analysis of their properties. For example it is
good for the study of their classical limit in the case of quantum groups, which is
crucial for quantization of conjugacy classes and their equivariant vector bundles  \cite{M2}.

A description of extremal vectors in Verma modules over Kac-Moody algebras is available in   \cite{MFF}
via a special calculus of polynomials with complex exponents. An inductive construction of extremal vectors in the
case of quantum groups was suggested in  \cite{KL}.
In a factorized form $\theta_{\bt,m}=\theta_{\bt,1}^m$, Shapovalov elements for classical Lie algebras are presented in \cite{Zh},  with the help of
extremal projectors \cite{AST}. While  Zhelobenko's construction gives an answer in the case of  classical simple Lie algebras,
 there remains a problem of explicit description of the structure of factors.

We suggest an alternative approach to the  problem based on a contravariant  bilinear form on Verma modules.
It is not as universal as Zhelobenko's, but has
its own advantages as it gives explicit expressions for Shapovalov elements in a factorized form in
almost all cases. There are only three roots in the exceptional Lie algebras $\g_2$, $\f_4$, and $\e_8$ that
are not covered by our method in its current version.

Extremal vectors generate the kernel of a canonical contravariant form on a Verma module,
which is a specialization of the "universal" Shapovalov form on the Borel subalgebra \cite{Shap}.
This form itself is extremely important and has numerous applications, see e.g. \cite{ES,EK,FTV,AL}.
For a generic weight, the Verma module is irreducible and  the form
is non-degenerate. The inverse form gives rise to an element $\Sc$
 of extended tensor product of positive and negative subalgebras of the (quantized) universal
 enveloping algebra.
 Sending the positive leg of $\Sc$ to a representation yields a matrix with entries in the negative
 subalgebra which we call   Shapovalov matrix.

Our approach consists in  relating $\theta_{\bt,m}$ with entries
of the   Shapovalov matrix, which is explicitly known for all
classical and quantum groups. It was obtained in \cite{M1} in  generalization of Nagel-Moshinski
expression for raising and lowering operators of $\s\l(n)$ \cite{NM}. It can also be
derived (in the quantum setting) from the ABRR equation on the dynamical twist \cite{ABRR,ES}.
 Our method provides not only factorization of $\theta_{\bt,m}$
to a product of (possibly shifted) $\theta_{\bt,1}$ but also an efficient recipe for description
of $\theta_{\bt,1}$ in a very elementary way, by a generalized Nagel-Moshinsky rule.

Our approach is absolutely parallel for a classical semi-simple Lie algebra $\g$
and its Drindeld-Jimbo quantum group. The classical case can be done directly or obtained as the limit case $q\to 1$
of the deformation parameter. Let us describe the  method in more detail.

With a finite dimensional module $V$ and a pair of non-zero
vectors $v,f_\bt v\in V$
we associate a Shapovalov matrix element  which belongs to the negative Borel subalgebra rationally extended over
 the Cartan subalgebra.
Under   certain assumptions  on $V$ and $v$, such  matrix elements deliver factors in $\theta_{\bt,m}$.
These factors normalize positive root vectors of the semi-simple subalgebra $\l\subset \g$
whose negative counterparts annihilate $v$. This way they become lowering operators in the Mickelsson algebras of the pair $(\g,\l)$,
\cite{Mick}.

The vector  $f_\bt v$ determines a
homomorphisms $V_{\la_2}\to V\tp V_{\la_1}$, where $V_{\la_i}$ are irreducible Verma modules of
highest weights $\la_i$ and $\la_2-\la_1$ is  the weight of $f_\bt v$.
Factorization of $\theta_{\bt,m}$ follows from factorization of the matrix element
of  the pair $v^{\tp m}$, $(f_\bt v)^{\tp m}$, and from the chain of homomorphisms
$$
V_{\la_{m}}\to V\tp V_{\la_{m-1}}\to \ldots \to  V^{\tp m}\tp V_{\la_0}.
$$
The vector  $v$ should be highest for the support of $\bt$ (the minimal simple subalgebra in $\g$ that accommodates $\bt$) and generate a $2$-dimensional submodule
of the subalgebra generated
by the root spaces $\g_{\pm \bt}$. These conditions are feasible for all $\bt$ but the three exceptional roots mentioned above.

As a result, we obtain $\theta_{\bt,m}(\la)$ as a product  $\prod_{i=0}^{m-1}\theta_{\bt,1}(\la_i)$ with $\la_0=\la$.
The factors $\theta_{\bt,1}$ can be calculated   by
the generalized Nagel-Moshinsky rule (\ref{norm_sing_vec}); that is done in the last section of the paper.
Viewed as an element of the Borel subalgebra,
$\theta_{\bt,m}$ becomes a product of shifted $\theta_{\bt,1}$, by the weight of $v$.
This shift degenerates to trivial
if $\bt$ contains a
simple root $\al$ of the same length with multiplicity $1$.
In this case $v$ can be chosen of weight $\omega_\al$, where $\omega_\al$ is
the corresponding fundamental generator of the weight lattice.
In that case, $\theta_{\bt,m}$ becomes a power of $\theta_{\bt,1}$.
The element $\theta_{\bt,1}$ is a Mickelsson generator for the pair $(\g,\l)$,
where  simple roots of $\l$
are complementary to $\al$. Different $\al$
results in different presentations.

Except for the last section, we present only the $q$-version of the theory. The classical case can be obtained by sending $q$ to $1$.
The expression for $\theta_{\bt,1}$ is  greatly simplified for $q=1$, so we give a special consideration to this case in
the last section.

\section{Preliminaries}

\label{SecPrelim}

Let  $\g$ be a  simple complex Lie algebra and  $\h\subset \g$ its Cartan subalgebra. Fix
a triangular decomposition  $\g=\g_-\op \h\op \g_+$  with  maximal nilpotent Lie subalgebras
$\g_\pm$.
Denote by  $\Rm \subset \h^*$ the root system of $\g$, and by $\Rm^+$ the subset of positive roots with basis $\Pi$
of simple roots. The basis $\Pi$ generates a root lattice $\Ga\subset \h^*$ with the positive semigroup $\Ga_+=\Z_+\Pi\subset \Ga$.

Choose an $\ad$-invariant form $(\>.\>,\>.\>)$ on $\g$, restrict it to $\h$, and transfer to $\h^*$ by duality.
For every $\la\in \h^*$ there is   a unique element $h_\la \in \h$ such that $\mu(h_\la)=(\mu,\la)$, for all $\mu\in \h^*$.
For a non-zero $\mu\in \h^*$  set  $\mu^\vee=\frac{2}{(\mu,\mu)}\mu$ and  $h_\mu^\vee=\frac{2}{(\mu,\mu)}h_\mu$.

Let  $\omega_\al$, $\al \in \Pi$ denote fundamental weights  determined by equations
$(\omega_\al, \bt^\vee)=\dt_{\al,\bt}$, for all $\al, \bt \in \Pi$.

Fix a non-zero complex number $q$ that is not a root of unity and
set  $[z]_q=\frac{q^{z}-q^{-z }}{q-q^{-1}}$ for $z\in \h+\C$.
The standard Drinfeld-Jimbo quantum group $U_q(\g)$ was introduced in \cite{D1,J}. It is a complex Hopf algebra with the set of generators $e_\al$, $f_\al$, and $q^{\pm h_\al}$, $\al \in \Pi$, satisfying relations
$$
q^{h_\al}e_\bt=q^{ (\al,\bt)}e_\bt q^{ h_\al},
\quad
[e_\al,f_\bt]=\dt_{\al,\bt}[h_\al]_q,
\quad
q^{ h_\al}f_\bt=q^{-(\al,\bt)}f_\bt q^{ h_\al},\quad \al, \bt \in \Pi.
$$
The elements $q^{h_\al}$ are invertible, with $q^{h_\al}q^{-h_\al}=1$, while  $\{e_\al\}_{\al\in \Pi}$ and $\{f_\al\}_{\al\in \Pi}$
also satisfy quantized Serre relations, see \cite{ChP} for details.

A Hopf algebra structure on $U_q(\g)$ is introduced by the comultiplication
$$
\Delta(f_\al)= f_\al\tp 1+q^{-h_\al}\tp f_\al,\quad \Delta(q^{\pm h_\al})=q^{\pm h_\al}\tp q^{\pm h_\al},\quad\Delta(e_\al)= e_\al\tp q^{h_\al}+1\tp e_\al
$$
set up on the generators and extended as an algebra  homomorphism $U_q(\g)\to U_q(\g)\tp U_q(\g)$.
The antipode is an algebra anti-automorphism of $U_q(\g)$ that acts on the generators by the assignment
$$
\gm( f_\al)=- q^{h_\al}f_\al, \quad \gm( q^{\pm h_\al})=q^{\mp h_\al}, \quad \gm( e_\al)=- e_\al q^{-h_\al}.
$$
The counit homomorphism $\eps\colon U_q(\g)\to \C$ returns
$$
\eps(e_\al)=0, \quad \eps(f_\al)=0, \quad \eps(q^{h_\al})=1.
$$
Denote by $U_q(\h)$,  $U_q(\g_+)$, and $U_q(\g_-)$  subalgebras   in $U_q(\g)$
 generated by $\{q^{\pm h_\al}\}_{\al\in \Pi}$, $\{e_\al\}_{\al\in \Pi}$, and $\{f_\al\}_{\al\in \Pi}$, respectively.
The quantum Borel subgroups are defined as $U_q(\b_\pm)=U_q(\g_\pm)U_q(\h)$; they are Hopf subalgebras in $U_q(\g)$.
We will also need their extended version $\hat U_q(\b_\pm)=U_q(\g_\pm)\hat U_q(\h)$, where
$\hat U_q(\h)$ is the ring of fractions of $U_q(\h)$ over the multiplicative system generated by
$[h_\al-c]_q$ with $\al \in \Gamma_+$ and $c\in \Qbb$.

We extend the notation $f_\al$, $e_\al$ to all  $\al\in \Rm^+$ meaning the Lusztig root vectors
with respect to some normal ordering of $\Rm^+$, \cite{ChP}. They are known to generate a Poincare-Birkhoff-Witt (PBW)
basis in $U_q(\g_\pm)$.

Given a $U_q(\g)$-module $V$, a non-zero vector $v$ is said to be of weight $\mu$ if $q^{h_\al}v=q^{(\mu,\al)} v$ for all $\al\in \Pi$.
The linear span of such vectors is denoted by $V[\mu]$.
A module $V$ is said to be of highest weight $\la$ if it is generated by vector $v\in V[\la]$ that
is killed by all $e_\al$. The vector $v$ is called highest; it is defined up to a non-zero scalar multiplier.

We consider an involutive coalgebra anti-automorphism and algebra automorphism $\si\colon U_q(\g)\to U_q(\g)$ setting it on the
generators by the assignment
$$\si\colon e_\al\mapsto f_\al, \quad\si\colon f_\al\mapsto e_\al, \quad \si\colon q^{h_\al}\mapsto q^{-h_\al}.$$
The involution $\omega =\gamma^{-1}\circ \si=\si \circ \gm$ is an algebra anti-automorphism and preserves
comultiplication.

A symmetric bilinear form $(\>.\>,\>.\>)$ on a $\g$-module $V$ is called contravariant if
$\bigl(\omega(x),y\bigr)=\bigl(x,\omega(y)\bigr)$ for all $x,y\in U_q(\g)$.
A module of highest weight has a unique $\C$-valued contravariant form such that the highest vector has squared norm $1$.
We call this form canonical and apply this term to the form on tensor products that is the product of canonical forms on tensor factors.
This is consistent because $\omega$ is a coalgebra map.

Let us recall the definition of the Shapovalov of $U_q(\h)$-valued Shapovalov form on the Borel subalgebra  $U_q(\b_-)$, \cite{Shap}.
Regard $U_q(\b_-)$ as a free right $U_q(\h)$-module generated by $U_q(\g_-)$.
The triangular decomposition $U_q(\g)=U_q(\g_-)U_q(\h)U_q(\g_+)$ facilitates projection
   $\wp\colon U_q(\g)\to U_q(\h)$ along the sum $\g_-U_q(\g)+U_q(\g)\g_+$, where $\g_-U_q(\g)$ and $U_q(\g)\g_+$ are
   right and left ideals generated by the negative and positive generators,
   respectively.
Set
$$
(x,y)=\wp\bigl(\omega(x)y\bigr), \quad x,y\in U_q(\g).
$$
This form is $U_q(\h)$-linear and contravariant.
It follows that the left ideal $U_q(\g)\g_+$ is in the kernel, so the form descends to a form
on the quotient  $U_q(\g)/U_q(\g)\g_+\simeq U_q(\b_-)$.

A Verma module $V_\la$ is an induced module $U_q(\g)\tp_{U_q(\b_+)}\C_\la$,
where $\C_\la$ is the 1-dimensional $U_q(\b_+)$-module that is trivial on $U_q(\g_+)$ and
returns  weight $\la$ on $U_q(\h)$. Its highest vector of weight $\la$  is denoted by  $v_\la$, which
is also called vacuum vector. It freely generates $V_\la$ as a module over $U_q(\g_-)$.

Specialization of the Shapovalov form at $\la\in \h^*$ gives the canonical contravariant $\C$-valued form $(x,y)_\la=\la\Bigl(\wp\bigl(\omega(x)y\bigr)\Bigr)$
on $V_\la$, upon a natural isomorphism $U_q(\g_-)\simeq V_\la$ of $U_q(\g_-)$-modules
generated by the assignment $1\mapsto v_\la$. Conversely, the canonical contravariant form
on $V_\la$ regarded as a function of $\la$ descends to the  Shapovalov form if one views $U_q(\h)$ as
the algebra of polynomial functions on $\h^*$.
By  an abuse of terminology, we also mean by Shapovalov form the canonical
contravariant form on $V_\la$.

It is known that the contravariant form on a Verma module goes degenerate if and only if
its highest weight is in the union of $\Hc_{\bt,m}=\{\la\>|\>q^{2(\la+\rho,\bt)-m(\bt,\bt)}=1\}$ over $\bt \in \Rm^+$ and $m\in \N$, where
$\rho$ is the half-sum of positive roots, \cite{DCK}.
In the classical case $q=1$, $\Hc_{\bt,m}$ becomes  a Kac-Kazhdan hyperplane of weights satisfying $2(\la+\rho,\bt)=m(\bt,\bt)$.

Recall that a vector $v\in V_\la$ of weight $\la-\mu$ with $\mu\in \Gamma_+$, $\mu\not =0$, is called extremal if $e_\al v=0$ for all $\al \in \Pi$.
We call its image under the isomorphism $V_\la\to U_q(\g_-)$ a  Shapovalov element. Extremal vectors   are
in the kernel of the contravariant form and generate submodules of the corresponding highest weights.
We will be interested in the special case when $\mu=m\bt$ with $\bt\in \Rm^+$ and $m\in \N$. Then the
highest weight $\la$ has to be in $\Hc_{\bt,m}$.

For simple $\bt$ the Shapovalov element $\theta_{\bt,m}$  is just the $m$-th power of
the root vector, $\theta_{\bt,m}=f_\bt^m$. That is not the case for compound $\bt$.
The goal of this work is to find explicit expressions for $\theta_{\bt,m}$ when  $\bt$ is  compound.

\section{Shapovalov inverse and its matrix elements}
Define the opposite $U_q(\g)$-module $V_\la'$ of lowest weight $-\la$ as follows. The underlying vector
space of $V_\la'$ is taken to be $V_\la$, while the representation homomorphism
$\pi'_\la$ is twisted by $\si$, that is $\pi'_\la=\pi_\la\circ \si$.
The module $V_\la'$ is freely generated over $U_q(\g_+)$ by its lowest vector $v_\la'$.

Let $\si_\la\colon V_\la\to V_\la'$ denote the isomorphism of vector spaces,
$
x v_\la\mapsto \si(x) v_\la', \quad x\in U_q(\g_-).
$
It intertwines the representations homomorphism, $\pi_\la' \circ \si= \si_\la\circ \pi_\la$.
The map $\si_\la$ relates the contravariant form on $V_\la$ with  an invariant pairing
$V_\la\tp V_\la'\to V_\la\tp V_\la \to \C$.

Suppose that the module $V_\la$ is irreducible. Then its invariant pairing is non-generate (as well as the contravariant form on $V_\la$).
The inverse form is an element of a completed tensor product $V_\la'\hat \tp V_\la$.
Under the isomorphisms $V_\la\to U_q(\g_-)$, $V_\la'\to U_q(\g_+)$, it goes to an element that we denote by $\Sc\in  U_q(\g_+)\hat \tp  U_q(\g_-)$
and call universal Shapovalov matrix.
Varying the highest weight $\la$ we get a rational trigonometric dependance of $\Sc$. As a function
of $\la$,  $\Sc$ is  regarded as an element of $U_q(\g_+)\hat \tp \hat U_q(\b_-)$. In other words,
the weight dependance is accommodated by the right tensor leg of $\Sc$.

Given a finite dimensional $U_q(\g)$-module with representation homomorphism
$\pi\colon U_q(\g)\to \End(V)$ the image $S=(\pi\tp \id)(\Sc)$ is a matrix with entries in $U_q(\g_-)$.
One can work directly with a rational trigonometric operator function  $S$  and forget that it came from $\Sc$.

An explicit expression of $S$ in a weight basis  $\{v_i\}_{i\in I}\subset V$ can be given in terms of  Hasse diagram
$\mathfrak{H}(V)$. Such a diagram is
associated with any partially ordered sets. Arrows are simple root vectors $e_\al$ connecting basis elements $v_j\stackrel{e_\al}{\longleftarrow} v_i$
whose weight difference is $\nu_j-\nu_i=\al$.
 We introduce a partial  order on $\{v_i\}_{i\in I}$  by writing $v_i\succ v_j$ if the inclusion $\nu_i-\nu_j\in \Gamma_+\backslash\{0\}$
 holds for their weights.
The matrix $S$ is triangular: $s_{ii}=1$ and $ s _{ij}=0$ if $\nu_i$ is not succeeding $\nu_j$. The entry $s_{ij}$
is a rational trigonometric function $\h^*\to U_q(\g_-)$. Its value carries weight $\nu_j-\nu_i\in -\Gamma_+$.

The matrix $S$ depends only on the $U_q(\b_+)$-module structure on $V$. Therefore, to calculate a matrix
element $s_{ij}$, one can choose a weight basis that extends a weight basis in the submodule $U_q(\g_+)v_j$.
Then, in particular, $s_{ij}=0$ if  $v_i\not \in U_q(\g_+)v_j$.

We  recall a construction of $S$  following \cite{M1}.
Let $\{h_i\}_{i=1}^{\rk \g}\in \h$ be an orthonormal basis. The element $q^{\sum_i h_i\tp h_i}$ belongs to a completion of $U_q(\h)\tp U_q(\h)$
in the $\hbar$-adic topology, where $\hbar  =\ln q$. Choose an R-matrix of $U_q(\g)$ such that $\check \Ru=q^{-\sum_i h_i\tp h_i}\Ru\in U_q(\g_+)\hat \tp U_q(\g_-)$
and set $\Cc=\frac{1}{q-q^{-1}}(\check \Ru- 1\tp 1)$.
The key identity on $\Cc$ that facilitates the $q$-version of the theory is  \cite{M1}
\be
[1\tp e_{\al},\Cc] +(e_{\al}\tp q^{-h_{\al}})\Cc- \Cc (e_{\al}\tp q^{h_{\al}}) =
e_{\al}\tp [h_{\al}]_q, \quad \forall \al \in \Pi^+.
\label{key_id}
\ee
In the classical limit, $\Cc=\sum_{\al\in  \Rm^+} e_\al\tp f_\al$ is the polarized split Casimir without its Cartan part.
One then recovers
\be
[1\tp e_{\al},\Cc]+[e_{\al}\tp 1,\Cc] = e_{\al}\tp h_{\al}
 \label{key_id_cl}
\ee
for each simple root $\al$.

We rectify the Hasse diagram and the partial ordering by removing arrows $v_i\leftarrow v_j$ if $C_{ij}=0$. This will not affect the formula (\ref{norm_sing_vec}) for matrix elements.

For each weight $\mu\in \Gamma_+$ put
\be
\eta_{\mu}=h_\mu+(\mu, \rho)-\frac{1}{2}(\mu,\mu) \in \h\op \C.
\ee
Regard it  as an affine function on $\h^*$ by  the assignment $\eta_\mu\colon \zt \mapsto (\mu,\zt+ \rho)-\frac{1}{2}(\mu,\mu)$, $\zt\in \h^*$.
Observe that $\eta_{m\bt}=m\bigl(h_\bt+(\bt, \rho)-\frac{m}{2}(\bt,\bt)\bigr)$. That is,  $[\eta_{m\bt}(\la)]_q$
vanishes on $\Hc_{\bt,m}$ (and only on the Kac-Kazhdan hyperplane in the classical case).

For a pair of non-zero vectors $v,w\in V$ define a matrix element $\langle w|v\rangle=(w,\Sc_1v)\Sc_2\in \hat U_q(\b_-)$,
where  $\Sc_1\tp \Sc_2$ stands for  a Sweedler-like notation for   $\Sc$ and the pairing is with respect to the canonical contravariant form on $V$.
Its specialization at a weight $\la$ is denoted by $\langle w|v\rangle_\la$.
For each $w$, the map $V\to V_\la$, $v\mapsto \langle v|w\rangle v_\la=\langle w|v\rangle_\la v_\la$ satisfies:
$e_\al\langle  v|w\rangle v_\la=\langle \si(f_\al) v|w\rangle v_\la$ for all $\al\in \Pi$.
This is a consequence of  $\U_q(\g_+)$-invariance of the element $\S(1\tp v_\la)\in U_q(\g_+)\hat \tp V_\la$.

Let $c_{ij}$ denote the entries of the matrix $(\pi\tp \id)(\Cc)$ in an orthonormal weight basis $\{v_i\}_{i\in I}\in V$.
The entries $s_{ij}$ of the matrix $S$ are the matrix elements $\langle v_i|v_j\rangle$.
Fix a "start" node $v_a$ and an "end" node $v_b$ such that $v_b\succ v_a$.
Then the re-scaled matrix element   $\check s_{ab}=-s_{ba}[\eta_{\nu_b-\nu_a}]_qq^{-\eta_{\nu_b-\nu_a}}$
can be calculated by the formula
\be
\check s_{ba}=c_{ba}+\sum_{k\geqslant 1}\sum_{v_b \succ v_k\succ \ldots \succ v_1\succ v_a}
c_{b k}\ldots
c_{1 a}\frac{(-1)^kq^{\eta_{\mu_k}}\ldots q^{\eta_{\mu_1}}}{[\eta_{\mu_k}]_q\ldots [\eta_{\mu_1}]_q} \in \hat U_q(\b_-),
\label{norm_sing_vec}
\ee
where
$
 \mu_l=\nu_{l}-\nu_{a}\in \Gamma_+$, $l=1,\ldots, k.
$
Here the summation is performed over all possible routes from $v_a$ to $v_b$, see \cite{M1} for details.

It is straightforward that $U_q(\g_+)$-invariance of the tensor $\Sc(v_a\tp v_\la)$ implies
\be
e_\al \check s_{ba}(\la)v_\la \propto [\eta_{\nu_b-\nu_a}(\la)]_q \sum_{k}\pi(f_\al)_{kb} s_{ka}(\la)v_\la.
\label{e-al-s_ij}
\ee
It follows that $\check s_{ba}(\la)v_\la$ is an extremal vector in $V_\la$ for $\la$ satisfying  $[\eta_{\nu_b-\nu_a}(\la)]_q=0$ provided
\begin{enumerate}
  \item $\check s_{ba}(\la)\not =0$,
  \item $\la$ is a regular point for all $s_{ka}(\la)$ and all  $\al$.
\end{enumerate}
We aim to find an appropriate matrix element for $\theta_{\bt,m}$ that satisfies these conditions.

    Let $V$ be a finite dimensional $U_q(\g)$-module with a pair of weight vectors  $v_a,v_b\in V$ such that $v_a=f_\bt v_b$
    for $\bt \in \Rm^+$.
    We call the triple $(V,v_a,v_b)$ a  $\bt$-representation.

\begin{propn}
\label{Shapdeg1}
  Let $(V,v_a,v_b)$ be a $\bt$-representation for  $\bt \in \Rm^+$. Then for generic $\la\in \Hc_{\bt,1}$
   the vector $\check s_{ba}(\la)v_\la\in V_\la$ is extremal.
\end{propn}
\begin{proof}
  The factors $\frac{q^{\eta_{\mu_k}}}{[\eta_{\mu_k}]_q}$  in (\ref{norm_sing_vec})  go singular on the union of a finite number of
  the null-sets $\{\la\>|\>[\eta_\mu(\la)]_q=0\}$. None of them coincides with
  the $\Hc_{\bt,1}$, hence $\check s_{ba}(\la)$ is regular at generic $\la\in \Hc_{\bt,1}$.
  By the same reasoning, all $s_{ka}(\la)$ are regular at such $\la$. Finally,
  the first term  $c_{ba}$ (and only this one) involves the Lusztig root vector $f_\bt$, a generator of a PBW basis
  in $U_q(\g_-)$. It is therefore independent of the
  other terms, and  $\check s_{ba}(\la)\not=0$.
\end{proof}

Upon identification of rational $U_q(\g_-)$-valued functions on $\h^*$ with $\hat U_q(\b_-)$ we
conclude that $\check s_{ba}$ is a Shapovalov element  $\theta_{\bt,1}$ and denote it by  $\theta_\bt$.
Uniqueness of extremal vector of given weight implies that all matrix elements $\check s_{ba}$ with $v_a=f_\bt v_\bt$ deliver the same $\theta_\bt$, up to a scalar factor. However, when we aim at $\theta_{\bt,m}$ with $m>1$, we have to choose matrix elements for $\theta_\bt$ more carefully.

It was relatively easy to secure the above two conditions in the case of $m=1$. For higher $m$ we will choose a different strategy:
we will satisfy the first condition by the very construction and bypass a proof of the second with different arguments.

\section{Factorization of Shapovalov elements}

For a positive root  $\bt\in \Pi$ denote by  $\Pi_\bt\subset \Pi$ the set of simple roots entering
the expansion of $\bt$ over the basis $\Pi$ with  positive coefficients.
A simple Lie subalgebra $\g(\bt)\subset \g$ generated by
$e_\al,f_\al$ with $\al\in \Pi_\bt$ is called  support of $\bt$. Its universal enveloping algebra is quantized as a Hopf subalgebra in $U_q(\g)$.
Since $[e_\mu,f_\nu]=0$ for $\nu\in \Pi_\bt$ and $\nu\in \Pi\backslash \Pi_\bt$, we can restrict to $\g=\g(\bt)$  without loss of generality.
\begin{definition}
  Let  $\bt \in \Rm^+$ be a positive root and $(V,v_a,v_b)$ a $\bt$-representation such that   $e_\al v_b=0$ for all $\al \in \Pi_\bt$
   and $(\nu_b,\bt^\vee)=1$. We call such $\bt$-representation admissible.
\end{definition}
 In other words, a triple  $(V,v_a,v_b)$ is  admissible if $v_b$ is the highest vector of a $U_q\bigl(\g(\bt)\bigr)$-submodule in $V$ that generates a 2-dimensional submodule of the $U_q\bigl(\s\l(2)\bigr)$-subalgebra generated by $f_\bt,e_\bt$.
It is clear that if a root has an admissible representation, then one of the fundamental representations is admissible.
It is also clear that $v_a$ and $v_b$ can be included in an orthonormal weight basis in $V$.

Assuming the triple $(V,v_a,v_b)$ admissible, denote by $V^{(m)}$ the irreducible (finite-dimensional) $U_q(\g)$-module with highest weight $m \nu_b$ and highest vector $v_b^m$.
The weight $m\nu_a$ is related with $m v_b$ by the reflection $\si_\bt\colon \bt \to -\bt$ from the Weyl group, hence $\dim V^{(m)}[m\nu_a]=1$. Set $v_a^m=f_\bt^m v_b^m\in V^{(m)}[m\nu_a]$.
This is a non-zero vector.

A factorization for $\theta_{\bt,m}$ we are seeking for is a consequence of the following factorization of Shapovalov matrix elements
in tensor product modules.

\begin{lemma}
\label{matr_factorization}
Let $(V,v_a,v_b)$ be an admissible $\bt$-representation. Then for all $m\in \N$,
  \be
\langle v_b^m| v_a^m\rangle_{\la_0}  = c\langle v_b|v_a\rangle_{\la_{m-1}}\ldots \langle v_b|v_a\rangle_{\la_0}
\label{factor-root-degree}
\ee
where $\la_k=\la+k\nu_a\in \h^*$ for $k \in \Z_+$ and $c$ is a non-zero scalar.
\end{lemma}
\begin{proof}
We realize $V^{(m)}$ as a submodule in the tensor product $V^{\tp m}$ generated by the highest
vector $v_b^m=v_b^{\tp m}$.
Let $\la$ be
such that all Verma modules $V_{\la_k}$ of highest weights $\la_k$, $k=0,\ldots, m-1$,  are irreducible and
consider a chain of module homomorphisms
$$
V_{\la_{m}}\to V\tp V_{\la_{m-1}}\to \ldots \to  V^{\tp m}\tp V_{\la_0}.
$$
They send the highest vectors  $v_{\la_k}\in V_{\la_k}$ to  the extremal vectors $\Sc(v_a\tp v_{\la_{k-1}})\in V\tp V_{\la_{k-1}}$.
The highest vector $v_{\la_m}$ goes over to   $w_m\tp v_{\la_0}$, where $w_m\in V^{\tp m}$ is   of weight $m\nu_a$. It is related with $v_a^{\tp m}$
by an invertible operator  from $\End(V^{\tp m})$, which is  $m-1$-fold dynamical twist \cite{ES}.

Pair  $\Sc(w_m\tp v_{\la_0})$   with $v_b^{\tp m}$ and calculate $\langle v_b^{\tp m}|w_m\rangle_{\la_0}$:
$$
\bigl(v_b^{\tp (m-1)},    \langle v_b| v_a \rangle_{\la_{m-1}}^{(1)}\Sc_{1} w_{m-1} \bigr) \> \langle v_b| v_a \rangle_{\la_{m-1}}^{(2)}\> \Sc_{2}(\la_{0})
=
\langle v_b| v_a \rangle_{\la_{m-1}}^{(2)}\Bigl\langle \omega\bigl(\langle v_b| v_a \rangle_{\la_{m-1}}^{(1)}\bigr)v_b^{\tp (m-1)}|w_{m-1}\Bigr\rangle_{\la_0},
$$
where we use the Sweedler notation $\Delta(x)=x^{(1)}\tp x^{(2)}\in U_q(\b_-)\tp U_q(\g_-)$ for the coproduct of $x\in U_q(\g_-)$.
Since $yq^{h_\al} v_b=\ve(y)q^{(\al,\bt)}  v_b$ for all $y \in U_q(\g_+)$ and $\al \in \Ga_+$,  we
arrive at
$$
\langle v_b^{\tp m}|w_m\rangle_{\la_0}=q^{-(\bt,\nu_b)}\langle v_b |v_a\rangle_{\la_{m-1}}\langle v_b^{\tp (m-1)}|w_{m-1}\rangle_{\la_0}.
$$
Proceeding by induction on $m$ we conclude that $\langle v_b^{\tp m}|w_m\rangle_{\la_0}$ equals the right-hand side of (\ref{factor-root-degree}),
up to the factor $q^{-m(\bt,\nu_b)}$.
Finally, we replace $w_m$ with its orthogonal projection to $V^{(m)}$, which is proportional to $v_a^m$ because $\dim V^{(m)}[m\nu_a]=1$.
This proves the lemma for generic and hence for all $\la$ where the right-hand side of (\ref{factor-root-degree}) makes sense.
\end{proof}
It follows from the above factorization that the least common denominator of the    extremal vector
$
u=\Sc (v^m_a \tp v_{\la})\in V^{(m)}\tp V_{\la}
$
contains $d(\la)=[\eta_\bt(\la+(m-1)\nu_a)]_q=[(\la+\rho)-\frac{m}{2}(\bt,\bt)]_q$. It comes from the leftmost factor $\langle v_b|v_a\rangle_{\la_{m-1}}$
in the right-hand side of (\ref{factor-root-degree}).
Denote by  $s_{v_b^m,v_a^m}(\la)$ the matrix element $\langle v_b^m| v_a^m\rangle_{\la}$.
Since $d$ divides $[\eta_{m\bt}]_q$, the re-scaled matrix element
$$\check s_{v_b^m,v_a^m}(\la)=c(\la)d(\la)s_{v_b^m,v_a^m}(\la)\propto  \prod_{k=0}^{m-1}\theta(\la_k),
$$ where $c(\la)=-q^{-\eta_{m\bt}(\la_{m-1})}\frac{[\eta_{m\bt}(\la)]_q}{d(\la)}$, is regular and does not vanish at generic $\la \in \Hc_{\bt,m}$
because $d(\la)$ cancels the pole in $\langle v_b|v_a\rangle_{\la_{m-1}}$.
Put  $\check u =d^k(\la) u$, where  $k\geqslant 1$ is the maximal degree of this  pole in $u$. It
is an extremal vector in $V^{(m)}\tp V_\la$ that is regular at generic $\la\in \Hc_{\bt,m}$.

\begin{thm}
\label{Prop_factoriztion}
  For generic $\la\in \Hc_{\bt,m}$,
    $\theta_{m,\bt}(\la) \propto \check s_{v_b^m,v_a^m}(\la)$.
\end{thm}
\begin{proof}
The vector $\check u$ is presentable as
$$
\check u= v^m_b \tp d^k(\la) v_{\la}+\ldots + v^m_a\tp d^{k-1}(\la)c(\la) \check s_{v_b^m,v_a^m}(\la)v_{\la}.
$$
We argue that
$\check u=   v^m_a\tp c(\la) \check s_{v_b^m,v_a^m}(\la)v_{\la}$ for generic $\la$ in $\Hc_{\bt,m}$, where $d(\la)=0$.
Indeed, the $V_{\la}$-components of $\check u$ span a $U_q(\g_+)$-submodule in $V_\la$ isomorphic to a quotient of the $V^{(m)}$-dual.
A vector of maximal weight in this span is extremal and  distinct from $v_{\la}$. But
$\theta_{\bt,m}(\la)v_{\la}$ is the only, up to a factor,  extremal vector in $V_{\la}$, for generic $\la$.
Therefore $k=1$ and $\theta_{\bt,m}\propto \check s_{v_b^m,v_a^m}$.
\end{proof}
 One can pass to the "universal form" of $\theta_{\bt}$ regarding it as an element of $\hat U_q(\b_-)$.  Then
\be
 \theta_{\bt,m}=  (\tau_{\nu_b}^{m-1}\theta_\bt) \>\ldots \> (\tau_{\nu_b}\theta_\bt)\>\theta_\bt,
\label{factor-root-degre-un}
\ee
where $\tau_\nu$ is an automorphism of $\hat U_q(\h)$ generated by the affine shift of $\h^*$ along the weight $\nu$,
that is,
$(\tau_\nu\varphi)(\mu) = \varphi(\mu+\nu)$, $\varphi\in \hat U_q(\h)$,   $\mu\in \h^*$.
One may ask when the shift is trivial, $\tau_{\nu_b}\theta_{\bt}=\theta_{\bt}$, and   $\theta_{\bt,m}$ is just the $m$-th power of $\theta_\bt$.
\begin{propn}
\label{plain-power}
  Let $\bt$ be a positive root. Suppose that there is $\al\in \Pi$ of the same length as $\bt$ that enters the expansion of $\bt$ over the basis $\Pi$ with multiplicity 1. Then $\theta_{\bt,m}=\theta_{\bt}^m\in U_q(\b_-)$.
\end{propn}
\begin{proof}
 Let $\l\subset \g(\bt)$ be the semi-simple subalgebra generated by simple root vectors $f_\mu,e_\mu$ with $\mu\not=\al$.
Take for  $V$ the fundamental  module of  highest weight $\omega_\al$.
Put     $v_b$ to be the highest vector and $v_a=f_\bt v_b$. Then $(V,v_a,v_b)$ is an admissible $\bt$-representation
because $(\nu_b, \bt)=(\omega_\al,\al)=\frac{(\al,\al)}{2}=\frac{(\bt,\bt)}{2}$.
We  write (\ref{norm_sing_vec}) as
$$
\theta_\bt(\la)=c_{ba}+\sum_{v_b\succ v_i\succ v_a} c_{bi}s_{ia}(\la).
$$
The highest vector $v_b$ is killed by $\l_-$, therefore the Hasse diagram between $v_a$ and $v_b$ looks
$$
v_b\quad\stackrel{e_{\al}}{\longleftarrow}\quad f_\al v_b\quad \ldots \quad  v_a,
$$
where arrows in  the suppressed part are simple root vectors from  $U_q(\l_+)$.
But then  the only copy of $f_\al$ is in  $c_{ia}$ while all $s_{ia}$ belong to $U_q(\l_-)\hat U_q(\h_\l)$,
the extended Borel subalgebra of $U_q(\l)$.

Finally, since $\Pi_\l$ is orthogonal to $v_b$, we have $(\mu, \nu_a)=-(\mu,\bt)$ for all $\mu\in \Rm^+_\l$.
Therefore
$$
\theta_\bt(\la_k)=\theta_\bt(\la-k\bt)
,\quad
\theta_{\bt,m}(\la)=\prod_{k=0}^{m-1}\theta_\bt(\la-k\bt),
$$
where the product is taken in the descending order from left to right.
This proves the plain power factorization because each $\theta_\bt$ carries weight $-\bt$.
\end{proof}

\begin{remark}
\em
Here are a few comments on the conditions of Lemma \ref{matr_factorization} and Proposition
\ref{plain-power}.
\begin{enumerate}
\item
In the case when $\g$ is one of the four classical types,  there is an admissible representation
for each compound root $\bt\in \Rm^+_\g$.
It can be realized in the minimal fundamental module for all $\bt$ except for a short
root of $\s\o(2n+1)$, in which case the fundamental spin module does the job,
\cite{M2}. The conditions of Proposition (\ref{plain-power}) hold in all these
cases except for a long root of $\s\p(2n)$.
In the latter case, one should take for $V$ the fundamental module of highest weight $\omega_\al$, where $\al$ is the simple long root.

   \item
If $\g$ is simply laced and $\bt$ contains a simple root $\al$ with multiplicity $1$, the fundamental module with highest
weight $\omega_\al$ satisfies  the conditions of Proposition \ref{plain-power}. That covers all roots of $\e_6$ and $\e_7$
and all but the maximal root of $\e_8$. If all multiplicities of simple roots in $\bt$ are 2 or higher
(as in the maximal root in $\e_8$),
there are not admissible $\bt$-representations.
\item
There are three roots that have no admissible representations:
\begin{itemize}
  \item a short root in  $\g_2$ that is a sum of one long  and two short simple roots,

  \item the root $\al_1+2\al_2+3\al_3+2\al_4$ in $\f_4$
  with respect to the enumeration
  $
\begin{picture}(85,15)
   \put(20,2){\circle{2}}  \put(21,2){\line(1,0){18}}
  \put(40,2){\circle{2}}
  \put(52,0){$>$}\put(41,3){\line(1,0){16}}\put(41,2){\line(1,0){16}}
  \put(60,2){\circle{2}}\put(61,2){\line(1,0){18}}
  \put(80,2){\circle{2}}
  \put(18,6){$\scriptstyle{1}$}
  \put(38,6){$\scriptstyle{2}$}
  \put(58,6){$\scriptstyle{3}$}
  \put(78,6){$\scriptstyle{4}$}
  \end{picture}
$,

  \item the maximal root of $\e_8$.
\end{itemize}
In the $\g_2$- and $\f_4$-cases, the only simple root with multiplicity $1$ in $\bt$ is longer than $\bt$.
In the simply laced case of $\e_8$, all simple roots enter maximal $\bt$ with multiplicities $\geqslant 2$.
  \item
It follows that a root may have a few plain power factorizations of its Shapovalov elements.
Say, if $\bt$ is a root of height $k$ for $\g=\s\l(n)$, then
$\theta_{\bt,m}$ admits $k$ (apparently different) presentations. This is in accordance with the results of Zhelobenko whose
presentations are parameterised by normal orderings of positive roots.

\end{enumerate}
\end{remark}
\section{Shapovalov elements of degree 1}
In this section we describe the factor   $\theta_{\bt}$
for a particular admissible $\bt$-representation $(V,v_a,v_b)$.
To a large extent it reduces to description of the relevant part of the Hasse diagram  $\mathfrak{H}(V)$.
We give a complete explicit solution in the classical case.
In the case of  $q\not =1$, we do it up to calculation of the entries of the matrix $\Cc$.

To that end, we need to figure out the Hasse sub-diagram
$\mathfrak{H}(v_b,v_a)\subset\mathfrak{H}(V)$ including all possible routes from $v_a$ to $v_b$.
We argue that $\mathfrak{H}(v_b,v_a)$ can be extracted from a diagram $\mathfrak{H}(\b_-)$
we associate with the adjoint representation $\simeq \g$ as follows.
The nodes of $\mathfrak{H}(\b_-)$ are elements of the Cartan-Weyl basis in $\b_-$:  $h^\vee_\al$, $\al\in \Pi$, and  $f_\mu$, $\mu\in \Rm^+$.
Here $h^\vee_\al=\frac{2}{(\al,\al)}h_\al$ so that $\al(h^\vee_\al)=2$.
Arrows   are $f_\mu \stackrel{e_\al}\longleftarrow f_\nu$  if $[e_\al,f_\nu]\propto f_\mu$ and
$h^\vee_\al \stackrel{e_\al}\longleftarrow f_\al$, $\al\in \Pi$.
For example, in the case  of  $\g=\g_2$ we have
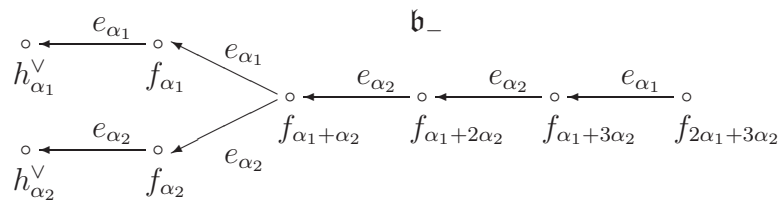
\begin{center}
\begin{picture}(310,70)
\put(155,60){$\b_-$}

\put(35,20){$e_{\al_2}$}
\put(85,10){$e_{\al_2}$}
\put(135,40){$e_{\al_2}$}
\put(185,40){$e_{\al_2}$}
\put(235,40){$e_{\al_1}$}

\put(5,0){$h^\vee_{\al_2}$}
\put(55,0){$f_{\al_2}$}
\put(105,20){$f_{\al_1+\al_2}$}
\put(155,20){$f_{\al_1+2\al_2} $}
\put(205,20){$f_{\al_1+3\al_2}$}
\put(255,20){$f_{2\al_1+3\al_2}$}

\put(10,15){\circle{3}}
\put(60,15){\circle{3}}
\put(110,35){\circle{3}}
\put(160,35){\circle{3}}
\put(210,35){\circle{3}}
\put(260,35){\circle{3}}

\put(55,15){\vector(-1,0){40}}
\put(105,34){\vector(-2,-1){40}}
\put(155,35){\vector(-1,0){40}}
\put(205,35){\vector(-1,0){40}}
\put(255,35){\vector(-1,0){40}}

\put(35,60){$e_{\al_1}$}
\put(85,50){$e_{\al_1}$}

\put(5,40){$h^\vee_{\al_1}$}
\put(55,40){$f_{\al_1}$}

\put(10,55){\circle{3}}
\put(60,55){\circle{3}}

\put(55,55){\vector(-1,0){40}}
\put(105,36){\vector(-2,1){40}}


\end{picture}
\end{center}
where   $\al_1$ is the long simple root and $\al_2$ is short.
This is a part of $\mathfrak{H}(\g)$ without arrows $\h^\vee_\al\leftarrow f_\mu $, $\al\not =\mu$.
For each $\al\in \Pi$ and $\bt\in \Rm^+$, the vector space underlying $\mathfrak{H}(h^\vee_\al,f_\bt)$ is a $\g_+$-module
(a subquotient of $\g$ by submodules generated by all $e_\mu$ and simple  $f_\mu$ with $\mu\not =\al$).

\begin{propn}
\label{Hasse_fund}
  Let  $V$ be a fundamental $\g$-module with highest weight $\omega_\al$ and highest vector $v$.
  Suppose that $(V,f_\bt v, v)$ is   an admissible representation for $\bt \in \Rm^+$. Then
  $\dim V[\mu]=1$ for each weight $\mu$ such that $\omega_\al \succeq \mu\succeq \omega_\al-\bt$.
  Furthermore, the assignment $h^\vee_\al\mapsto v$, $f_\gm\mapsto f_\gm v$, $\gm=\omega_\al-\mu$, induces an isomorphism
  of Hasse sub-diagram $\mathfrak{H}(h^\vee_\al, f_\bt)\subset \mathfrak{H}(\b_-)$ and
  $\mathfrak{H}(v, f_\bt v)\subset \mathfrak{H}(V)$.
\end{propn}

\begin{proof}
It is sufficient to prove it for the case $q=1$ because the weight structure is independent of $q$.
First of all, observe that $f_\gm v\not =0$ once $\gm\preceq \bt$, because $(\gm,\omega_\al)\not =0$.
We are left  to prove that $\psi_\gm v\propto f_\gm v$ for each $\psi\in U(\g_-)$ of weight $\gm\in \Rm^+$.

  Let $\ell\in \N$ be the multiplicity of $\al$ in $\bt$, then $\ell(\al,\al)=(\bt,\bt)$.
  Suppose that $\ell=1$.
  The weight  subspace $V[\omega_\al-\gm]$ with $\omega_\al\succ \omega_\al-\gm\succ \omega_\al-\bt$ is constructed as follows.
  For $\gm=\al$ it is $f_\al v$. Let  $\psi\in U(\g_-) $ be of weight $\gm$ which contains $\al$ with multiplicity $1$.
  Then   $\psi v\propto f_\gm v$
  if $\gm\in \Rm^+$   and  $\psi v=0$ otherwise, because
  all simple root vectors other than $f_\al$ annihilate $v$ and every monomial in $\psi$ can be replaced with a composition of commutators.
  The same is true in the case $\ell=2$.  Indeed, it is sufficient to check it for
   $f_\al \psi_\gm f_\al v$ where $\psi_\gm$ a monomial in $f_\si$, $\si\not =\al$
  of weight $\gm\in \Rm^+$.
  As we already proved, $\psi_\gm f_\al v$ is proportional to $[f_\gm,f_\al]v$. But then
  $f_\al \psi_\gm f_\al v\propto f_\al [f_\gm,f_\al]v=2 f_\al f_\gm f_\al v=  [f_\al,[ f_\gm, f_\al]]v$ because
  $f_\gm v =0= f_\al^2v$.
Finally,
  the case $\ell=3$ occurs only in $\g=\g_3$ (in the roots $\bt=\al_1+3\al_2$ and  $\bt =2\al_1+3\al_2$). Then $\al=\al_2$, $\omega_\al=\al_1+2\al_2$,  $V=\C^7$, and the statement  can be checked directly.

  Thus  there is a linear bijection between $\C h^\vee_\gm + \sum_{\al\preceq\gm\preceq \bt} \g_{-\gm}\subset \g$
  and $\C v  + \sum_{\al\preceq\gm\preceq\bt} \g_{-\gm} v\subset V$ that shifts weights   by $\omega_\al$. It is determined
  by the assignment on the basis as stated.
   It is easy to see that it is an isomorphism
  of $\g_+$-modules $\g(h^\vee_\al, f_\bt)$ and $V(v,f_\bt v)$ which induces an isomorphism
  of the corresponding sub-diagrams.
\end{proof}
Specialization of the formula (\ref{norm_sing_vec}) for $\theta_{\bt}$ in the light of Proposition \ref{Hasse_fund}
requires the knowledge of   matrix  $(\pi\tp \id)(\Cc)\in \End(V)\tp U_q(\g_-)$. We do it for the case of  $q=1$.
Let $\bt\in \Rm^+$ and $\al\in \Pi$ be such as in the above proposition.
For $\nu,\gm\in \Rm^+$, denote by  $C_{\nu,\gm}\in \C$ the scalars such that
$[e_\nu,f_\gm]=C_{\nu, \gm}f_{\gm-\nu}$,
if $\gm-\nu\in \Rm^+$,
$C_{\gm,\gm}=\frac{(\bt,\bt)}{2}\frac{\ell_{\al,\gm}}{\ell_{\al,\bt}}$, and
$C_{\nu,\gm}=0$ otherwise.
Then
$$
(\pi\tp \id)(\Cc)(f_\gm v_b\tp 1)=C_{\gm,\gm}v_b\tp f_\gm+\sum_{\nu\prec \gm} C_{\nu,\gm}f_{\gm-\nu}v_b\tp f_\nu.
$$
for all $\gm$ satisfying $\al\preceq \gm\preceq \bt$.
The formula (\ref{norm_sing_vec}) becomes
\be
\theta_{\bt}=
C_{\bt,\bt} f_{\bt}+\sum_{k\geqslant 1}\sum_{\nu_1+\ldots+\nu_{k+1}=\bt}
(C_{\nu_{k+1},\gm_{k}}\ldots C_{\nu_{1},\gm_0})(f_{\nu_{k+1}}\ldots
f_{\nu_1})\frac{(-1)^k }{\eta_{\mu_k}\ldots \eta_{\mu_1}}.
\label{norm_sing_vec-1}
\ee
The internal summation is performed over all partitions
of  $\bt$ to a sum of positive roots $\nu_i$ subject to the following:
all   $\gm_i=\gm_{i-1}-\nu_i$ for $i=1,\ldots, k$ with $\gm_0=\bt$
are positive roots and $\gm_i\succeq \al$. The weights $\mu_i$ are defined as $\mu_i=\gm_0-\gm_i=\nu_1+\ldots+\nu_i$.
 Note that in the $q\not =1$ case this sum  may involve terms with matrix entries of $\Cc$
whose weights are not roots.

\vspace{10pt}
\noindent
\underline{\large \bf Acknowledgement.}
\vspace{10pt} \\
 This  research   was  supported by a grant for creation
and development of International Mathematical Centers, agreement no.
075-15-2019-16-20 of November 8, 2019, between Ministry of Science and
Higher  Education of Russia and PDMI RAS.
We are grateful to Vadim Ostapenko for stimulating discussions.



 \end{document}